\documentclass{article}

\usepackage{amsmath,amsthm,amssymb, graphicx}
\usepackage{mathrsfs, stmaryrd}
\usepackage[polutonikogreek,english]{babel}
\usepackage[colorlinks,allcolors=black]{hyperref}
\usepackage{paracol}
\usepackage{tikz-cd}
\usepackage{tikz,fp}
\usepackage[all]{xy}
\xyoption{2cell}
\UseAllTwocells
\usetikzlibrary{arrows.meta}
\usetikzlibrary{matrix,arrows,decorations.pathmorphing,positioning}
\usetikzlibrary{decorations.markings,fixedpointarithmetic,patterns,arrows.meta}

\author{Georgios Chara-Lambous\thanks{\href{e-mail:\;}{math.georgios@gmail.com}}\\
\small DPMMS, University of Cambridge\\
}

\title{On the relationship between Galois and Tannakian categories: an open letter}

\makeatletter
\renewcommand\paragraph{\@startsection{paragraph}{4}{\z@}%
            {-2.5ex\@plus -1ex \@minus -.25ex}%
            {1.25ex \@plus .25ex}%
            {\normalfont\normalsize\bfseries}}
\makeatother
\newtheorem{definition}{Definition}[section]
\newtheorem{theorem}[definition]{Theorem}
\newtheorem{remark}[definition]{Remark}
\newtheorem{remarks}[definition]{Remarks}

\newtheorem{proposition}[definition]{Proposition}

\DeclareMathOperator{\lin}{lin}
\let\hom\relax
\DeclareMathOperator{\hom}{Hom}
\DeclareMathOperator{\stone}{\bf{Stone}}
\DeclareMathOperator{\pierce}{\bf{Pierce}}
\DeclareMathOperator{\ult}{Ult}
\DeclareMathOperator{\clopen}{Clopen}
\DeclareMathOperator{\calg}{\mathbf{CAlg}}
\DeclareMathOperator{\mon}{mon}
\DeclareMathOperator{\comon}{comon}
\DeclareMathOperator{\Aut}{\mathfrak{AUT}}

\DeclareMathOperator{\aut}{Aut}
\DeclareMathOperator{\proet}{pro\acute{e}t}
\DeclareMathOperator{\grp}{\mathbf{Grp}}
\DeclareMathOperator{\cogrp}{\mathbf{Cogrp}}

\DeclareMathOperator{\chopf}{\mathbf{CHopf}}
\DeclareMathOperator{\bool}{\mathbf{Bool}}
\DeclareMathOperator{\Rep}{\mathbf{Rep}}
\DeclareMathOperator{\Vect}{\mathbf{Vec}}
\DeclareMathOperator{\comod}{\mathbf{Comod}}
\DeclareMathOperator{\Comod}{\mathfrak{COMOD}}
\DeclareMathOperator{\stonegrp}{\mathbf{StoneGrp}}
\DeclareMathOperator{\cmon}{\mathbf{CMon}}
\DeclareMathOperator{\gal}{\mathbf{Gal}}
\DeclareMathOperator{\ob}{ob}

\DeclareMathOperator{\fd}{f.d}
\DeclareMathOperator{\fg}{f.g}
\DeclareMathOperator{\tann}{\mathbf{Tan}}
\DeclareMathOperator{\ntan}{\mathbf{NTan}}
\DeclareMathOperator{\affgrpsch}{\mathbf{AffGrpSch}}
\DeclareMathOperator{\affgrp}{\mathbf{AffGrp}}
\DeclareMathOperator{\alggrp}{\mathbf{AlgGrp}}
\DeclareMathOperator{\alggrpsch}{\mathbf{AlgGrpSch}}
\DeclareMathOperator{\etgrp}{\mathbf{\acute{E}tGrp}}
\DeclareMathOperator{\etgrpsch}{\mathbf{\acute{E}tGrpSch}}

\DeclareMathOperator{\sch}{\mathbf{Sch}}

\DeclareMathOperator{\fet}{\mathbf{F\acute{E}t}}
\DeclareMathOperator{\idemp}{Idemp}
\DeclareMathOperator{\ta}{Tan}
\DeclareMathOperator{\disc}{disc}
\DeclareMathOperator{\ga}{Gal}

\DeclareMathOperator{\cont}{Cont}
\DeclareMathOperator{\End}{End}

\DeclareMathOperator{\Cont}{\mathfrak{CONT}}
\DeclareMathOperator{\pro}{\mathbf{Pro}}
\DeclareMathOperator{\ind}{\mathbf{Ind}}

\DeclareMathOperator{\set}{\mathbf{Set}}
\DeclareMathOperator{\cring}{\mathbf{CRing}}
\DeclareMathOperator{\aff}{\mathbf{Aff}}

\DeclareMathOperator{\Mod}{\mathbf{Mod}}
\DeclareMathOperator{\Proj}{\mathbf{Proj}}

\DeclareMathOperator{\vv}{\mathcal{V}}

\DeclareMathOperator{\kernel}{ker}

\DeclareMathOperator{\id}{id}
\DeclareMathOperator{\et}{\acute{e}t}
\DeclareMathOperator{\spec}{Spec}

\DeclareMathOperator{\op}{op}
\DeclareMathOperator{\csep}{CSep}
\DeclareMathOperator{\fsep}{\mathbf{FCSep}}
\DeclareMathOperator{\ff}{\mathbb{F}}
\renewcommand{\c}{\mathcal{C}}

\renewcommand{\tt}{\mathbb{T}}

\begin{document}

\maketitle

{\centering\footnotesize To my brother \textgreek{Qr`istos} and his wife \textgreek{Mar`ia} on the occasion of their wedding.\par}

\begin{abstract}
It has long been said that the theories of Galois and Tannakian categories over a field $k$ are just ``formally similar''. With this note I will argue that this is in fact not the case: not only do Tannakian categories generalize Galois categories, but in fact the latter sit inside the former in a rather structured way. In particular, in the first part, I will claim that at the core of the general connection between $k$-Tannakian and Galois categories lies a $2$-functorial machine $$\tt\mapsto \Pi(\tt),$$  which takes as input any $k$-Tannakian category $\tt$ and associates to it a $k$-Tannakian category $\Pi(\tt)$ which, when $k$ is a separably closed field, determines and is determined by a Galois category $$\csep(\Pi(\tt)),$$ where the operation $\csep(-)$ is defined using the notion of \emph{commutative separable algebra in a symmetric monoidal category}, that Aurelio Carboni defines and uses in the beautiful paper \cite{carboni} to construct a Boolean pretopos out of a compact closed additive category with equalizers. 
\par This machine behaves, in the general case where no fibre functor is chosen on the categories, as a ``$2$-dimensional mono-coreflector'' when $k$ is separably closed. In the second part of this work I will demonstrate this in the simpler case where the categories are neutralized, and in which the results of \cite{carboni} are readily available, by arguing that the ($2$-) category $\gal_{*}$ of neutralized Galois categories is ($2$-) coreflective 
\[\begin{tikzcd}[row sep=large]
\ntan_{k,*}\arrow[d,shift left=1ex, "\gamma"]\\
\gal_{*}\arrow[u,hook,shift left=1.4ex, "\dashv"']
\end{tikzcd}
\]
in the ($2$-) category $\ntan_{k,*}$ of neutralized $k$-Tannakian categories. 
\end{abstract}

\newpage

\section*{Confiteor}
I would like to apologize to everyone who has reached out to me enquiring with interest about this (many years old) work, aspects of which I have presented both in Cambridge and at the PSSL 101 (Leeds) in 2017. I am extremely fascinated by the Tannakian formalism and the current raison d'   \^{e}tre of my academic life is to explore how widely applicable it is - as I strongly believe that the formalism's underlying mechanics are useful much more broadly than it has been recognized. While working on this more ``fundamental'' project, the work presented here was put aside, mainly because the ``complete demystification" of the connection between Galois and Tannakian categories ended up being a task bigger than I initially anticipated. And even though I had na\"{i}vely hoped to explore every aspect of it alone before sharing my thoughts, the emotional weight that this attitude carries has now become unbearable; and I am thus opening Pandora's box, in the hope that I will start making more material public. As a compromise, even though this article does present the main idea behind the connection, I will treat it as a ``working'' paper: I am planning to revisit it with additions and more details, as the state of it is by no means complete. 
\par The worlds of Tannakian and Galois categories are beautiful gardens bearing fruits from many different areas of mathematics, and I have tried to demonstrate this by using rather diverse language here - the first part is written with Algebraic Geometers in mind while the second is written with Category Theorists in mind. In fact I see the exploration of the connection between Galois and Tannakian categories as a great cooperation opportunity between these communities: in my opinion insight from both sides is necessary in order to harvest the full potential of this relationship. I welcome any questions, suggestions and feedback on this work, but most importantly (the reason for deciding to publish this in a working paper state) I would like to invite anyone who has thoughts and insights, either from a geometric or a categorical perspective, to write about them, as I really believe there are aspects that I have missed but also fruitful directions that deserve to be developed independently. Studying this connection seriously is a few decades overdue, and this has to be a collaborative effort.

\section*{Acknowledgments}
Firstly I would like to thank Martin Hyland for all his support: for accepting my suggestion to study the connection between Galois and Tannakian categories as a research project in the first place, for allowing me complete freedom in my work, for not losing hope in me regardless of my secretive attitude, and for his numerous attempts at convincing me to write my thoughts down in detail. I would also like to thank Pierre Deligne who, through his works and our correspondence, has taught me how to use the Tannakian formalism in a ``working'' way - the interpretation of my adjunction using ``Tannakian Algebraic Geometry'' is a result of his guidance. Finally, I would like to thank Peter Johnstone for his interest and encouragement for this work in the early stages. 

\newpage

\part{}
In this first part I will use the language of gerbes and Deligne's theory of ``Algebraic Geometry internal in a Tannakian category'' to describe the connection between Tannakian and Galois categories ``both internally and externally'' - what this means will be made clear in the body of the text.
\section{The general picture}
\subsection{The $1$-dimensional situation}
Let me start by discussing the $1$-dimensional situation that the aforementioned procedure ``categorifies''. Let $k$ be a field and consider the inclusion 
\[\begin{tikzcd}[row sep=large]
\alggrpsch_k\\
\etgrpsch_k\arrow[u,hook, "\iota"]
\end{tikzcd}
\]
of the full subcategory of \'{e}tale (algebraic) group $k$-schemes in the category of (affine) algebraic group $k$-schemes. Every algebraic group $k$-scheme $G$ induces the short exact sequence of algebraic group $k$-schemes
\[
\begin{tikzpicture}
\node (A) at (-0.8,0) {$0$};
\node (B) at (1.5,0) {$G^0:=\ker q_G$};
\node (C) at (3.8,0) {$G$};
\node (D) at (5.8,0) {$\pi_0(G)$};
\node (E) at (7.6,0) {$0$};

\draw[->] (A) to (B);
\draw[->] (B) to (C);
\draw[->] (C) to (D);
\draw[->] (C) to node (g) {} node[label=above:$\scriptstyle q_G$] {} (D);
\draw[->] (D) to (E);
\end{tikzpicture}
\]
where
\begin{itemize}
\item the (finite) \'{e}tale group $k$-scheme $\pi_0(G)$ \emph{of connected components} is defined as $$\pi_0(G):=\spec L,$$ with $L$ the largest separable $k$-subalgebra of the finitely generated commutative Hopf $k$-algebra $\mathcal{O}(G)$, and with the map $q_G$, induced by the inclusion $$L\subseteq\mathcal{O}(G),$$ being a surjective map of (affine) group $k$-schemes courtesy of the fact that $\mathcal{O}(G)$ is faithfully flat over $L$,
\item and the \emph{connected component} or \emph{identity component} algebraic group $k$-scheme $G^0$, defined as the kernel of $q_G$, is connected.
\end{itemize}
\par This construction of $\pi_0(G)$ is functorial, and the induced functor $\pi_0$ is a left adjoint to the inclusion $\iota$, with every $q_G$ being an epimorphism (a surjective map of affine group $k$-schemes is easily seen to be an epimorphism in $\alggrpsch_k$), realizing $\etgrpsch_k$ as a full, epi-reflective subcategory 
\[\begin{tikzcd}[row sep=large]
\alggrpsch_k\arrow[d,shift left=1.4ex, "\pi_0"]\\
\etgrpsch_k\arrow[u,hook,shift left=1.4ex, "\iota", "\vdash"']
\end{tikzcd}
\]
of $\alggrpsch_k$. 
\par The next observation is that this reflection $\pro$-completes: indeed, every affine group $k$-scheme $G$ can be written as a cofiltered limit $$G\cong\varprojlim_i G_i$$ of algebraic $k$-group schemes $G_i$, with each $G_i=\spec \mathcal{O}(G)_i$ corresponding to the finitely generated commutative Hopf $k$-subalgebras $\mathcal{O}(G)_i\subseteq \mathcal{O}(G)$ realizing $$\mathcal{O}(G)_i\cong\varinjlim_i \mathcal{O}(G)_i,$$ as the filtered union of the $\mathcal{O}(G)_i$. In fact, $$\chopf_k\simeq\ind\chopf_{k,\fg}$$ and $$\affgrpsch_k\simeq \pro\alggrpsch_k.$$ 
Taking the group of connected components of every $G_i$ induces a canonical morphism $$q_G:G\cong\varprojlim_i G_i\to\varprojlim_i \pi_0(G_i)=:\pi_0(G),$$ to a pro- (finite) \'{e}tale group $k$-scheme whose kernel identifies with $$\ker q_G\cong \varprojlim_i G_i^0,$$ and it is connected. 
\begin{remark}
Notice that I am abusively using the notation $\pi_0(G)$ for what really is $\pro\pi_0 (G)$; similar abuse for the quotient map notation, but I think the context makes everything clear. 
\end{remark}
This procedure of ``taking the maximal pro-\'{e}tale quotient of an affine $k$-group scheme'' induces an epi-reflection 
\[\begin{tikzcd}[row sep=large]
\affgrpsch_k\arrow[d,shift left=1.4ex, "\pi_0"]\\
\affgrpsch_k^{\proet}\arrow[u,hook,shift left=1.4ex, "\iota", "\vdash"']
\end{tikzcd}
\]
where $\iota$ is now the inclusion in $\affgrpsch_k$ of the full subcategory $\affgrpsch_k^{\proet}$ consisting of those affine $k$-group schemes that can be written as cofiltered limits of \'{e}tale group $k$-schemes.
\begin{remark}[``Internal'' Recovery]
It is a well-known fact that the counit of a reflection is a natural isomorphism. In this way we can ``recover'' every pro-\'{e}tale group $k$-scheme $E$ through a canonical isomorphism $\epsilon_E:\pi_0\circ\iota (E)\cong E$ induced by the counit of the adjunction $\pi_0\dashv \iota$. Notice that the ``reconstruction'' happens ``internally'', or ``within'' the category of affine group $k$-schemes.
\end{remark}
\par Let's see now how the situation ``externalizes'' by picking a separable closure of $k$: let $\overline{k}$ be one such separable closure of $k$. As it can be gathered from the upcoming Section 3.2, the assignment $$G\mapsto G(\overline{k}),$$ sending a finite \'{e}tale $k$-scheme (or \emph{finite \'{e}tale cover}, in the terminology of Section 3.2) to its (``rational'') $\overline{k}$-points induces an equivalence $$\fet_k\simeq \cont_f(\gal(\overline{k}/k)),$$ between the category of finite \'{e}tale $k$-schemes and the category of discrete finite sets endowed with a continuous action of the profinite (Stone) absolute Galois group of $k$. Choosing now $k$ to be separably closed kills the absolute Galois group, $\cont_f(\gal(\overline{k}/k))$ identifies with the category $\set_f$ of finite sets and by passing to the group objects we end up with an equivalence 
\begin{eqnarray*}
\etgrpsch_k&\to&\grp_f\\
G&\mapsto& G(k).
\end{eqnarray*}
By composing with $\pi_0$, we get a functor
\begin{eqnarray*}
\alggrpsch_k&\to&\grp_f\\
G&\mapsto& \pi_0(-)_k:=\pi_0(-)(k).
\end{eqnarray*}
which is left adjoint to the fully faithful functor 
\begin{eqnarray*}
\grp_f&\to&\alggrpsch_k\\
G&\mapsto& k^G,
\end{eqnarray*}
where $k^G$ is the coproduct of $|G|$-many copies of $\spec k$; i.e. we have a reflection 
\[\begin{tikzcd}[row sep=large]
\alggrpsch_k\arrow[d,shift left=1.4ex, "\pi_0(-)_k"]\\
\grp_f\arrow[u,hook,shift left=1.4ex, "k^{(-)}", "\vdash"']
\end{tikzcd}
\]
which $\pro$-completes to give an adjunction (the ``external'' incarnation of the reflection $\pi_0\dashv \iota$ of $\affgrpsch_k^{\proet}$ in $\affgrpsch_k$ for $k$ separably closed), 
\[\begin{tikzcd}[row sep=large]
\affgrpsch_k\arrow[d,shift left=1.4ex, "\pro \pi_0(-)_k"]\\
\stonegrp\arrow[u,hook,shift left=1.4ex, "I", "\vdash"']
\end{tikzcd}
\]
where I have replaced (as a harbinger to a connection with the Pierce spectrum I will mention later) the expected functor $\pro k^{(-)}$ with the functor $I$ defined as follows,
\begin{eqnarray*}
I:\stonegrp&\to&\affgrpsch_k\\
\pi&\mapsto& \spec \cont(\pi,k_{\disc}),
\end{eqnarray*}
with $\cont(\pi,k_{\disc})$ being the commutative Hopf $k$-algebra of continuous functions from $\pi$ to the field $k$ endowed with the discrete topology.
\begin{remark}[``External'' Recovery]\label{recext}
We have a canonical isomorphism $\epsilon_\pi:\pi_0\circ I(\pi)\cong \pi$ for every Stone group $\pi$.
\end{remark}

\subsection{The $2$-dimensional situation}
Now let's look at general Tannakian categories over a field $k$ and Galois categories collectively. Write $\tann_k$ and $\gal$ for their $2$-categories respectively.
\begin{remark}
Notice that both Tannakian and Galois categories are considered in this section without any particular fibre functors attached.
\end{remark}
\par I would like the reader to think of the situation here as similar in spirit to its $1$-dimensional counterpart discussed in the previous subsection. There, I argued that for $k$ separably closed, one can associate to a Stone group $\pi$ an affine (pro-\'{e}tale) group $k$-scheme $I(\pi)$ and from it one can go back to (``recover'') $\pi$ canonically (cf. Remark \ref{recext}). The point is that the same holds, in an appropriately $2$-dimensional way, for Galois and Tannakian categories: for any field $k$, one can associate, uniquely, a Galois category $\c$ (without any chosen fibre functor) a (neutral) Tannakian category $\c_{\lin}$, and when $k$ is separably closed, one can in fact recover $\c$ from $\c_{\lin}$. In this case, $\c$ and $\c_{\lin}$ have ``the same'' fibre functors.  
\begin{remark}
Let me elaborate on the obstruction (without going into details in its cohomological nature due to the fact that the space of fibre functors is a space of torsors) to doing this for a general $k$. Firstly notice that for every two fibre functors $p_1,p_2:\c\to\set_f$ on $\c$, we have $$\pi_1^{\ga}(\c,p_1)\cong\pi_1^{\ga}(\c,p_2).$$ However, if $k$ is not assumed to be separably closed, it is possible that $\c_{\lin}$ admits fibre functors $\omega_1,\omega_2:\c_{\lin}\to\Vect_k^{\fd}$ on $\c_{\lin}$ with values in $k$, for which $$\pi_1^{\ta}(\c_{\lin},\omega_1)\ncong\pi_1^{\ta}(\c_{\lin},\omega_2).$$ In particular it is possible that for one such fibre functor $\omega$, no Stone group $\pi\cong\varprojlim_i G_i$ exists with $$\pi_1^{\ta}(\c_{\lin},\omega)\cong \varprojlim_i k^{G_i},$$ i.e. $\c_{\lin}$ can admit fibre functors that are not recognizable by affine group $k$-schemes that come from Stone groups, which is, let me stress again, contrary to what happens for $k$ separably closed.
\end{remark}
Now let me describe how the (co-\footnote{Notice that the $2$-equivalence between Tannakian $k$-categories and affine fpqc gerbes over $\spec k$ is \emph{contravariant}, so it is \emph{co-} reflective behaviour that should be expected on this side of the theory.}) reflection is actually constructed - there are, again, internal and external ways to look at this. To understand the effect of the $2$-coreflection internally - again I use this phrase just to signify that everything happens within the $2$-category $\tann_k$ - let $\tt$ be a Tannakian category over $k$ and write $\pi(\tt)$ for the fundamental group of $\tt$, as constructed in Section 8 of \cite{deligne}; it is an affine group $\tt$-scheme, $$\pi(\tt)\in\affgrp_{\tt}:=\grp(\cmon(\ind\tt)^{\op})$$ which acts on every object of $\tt$. Now notice that the construction of the $\pi_0(G)$ of an affine group $k$-scheme $G$ is compatible with base change, and thus we can form the internal incarnation $$\pi_0(\pi(\tt))$$ of $\pi_0$ (see Exemple 8.14 of \cite{deligne}), by performing the construction of the previous subsection on the ``realized'' affine group $k$-schemes $$\pi_0(\pi(\tt))_{\omega}$$ for every fibre functor $\omega$. This procedure induces a faithfully flat morphism $$q_{\pi(\tt)}:\pi(\tt)\to \pi_0(\pi(\tt))$$ in $\affgrp_{\tt}$. The kernel $$\pi(\tt)^0:=\ker q_{\pi(\tt)}$$ is a normal affine group sub-$\tt$-scheme of $\pi(\tt)$, and in fact we get a short exact sequence of affine group $\tt$-schemes
\[
\begin{tikzpicture}
\node (A) at (-0.8,0) {$0$};
\node (B) at (1.5,0) {$\pi(\tt)^0$};
\node (C) at (3.8,0) {$\pi(\tt)$};
\node (D) at (5.8,0) {$\pi_0(\pi(\tt))$};
\node (E) at (7.6,0) {$0$};

\draw[->] (A) to (B);
\draw[->] (B) to (C);
\draw[->] (C) to (D);
\draw[->] (C) to node (g) {} node[label=above:$\scriptstyle q_{\pi(\tt)}$] {} (D);
\draw[->] (D) to (E);
\end{tikzpicture}
\] 
The normal group sub-$\tt$-scheme $\pi(\tt)_0$ of $\pi(\tt)$ induces a Tannakian subcategory $$\Pi(\tt)\subseteq\tt$$ generated by the objects of $\tt$ on which the connected component $\pi(\tt)^0$ acts trivially (cf. \cite{deligne89}, 6.6), which, from the gerbe perspective, corresponds to an fpqc gerbe whose band is a profinite group $k$-scheme. 
\par Assuming $k$ to be separably closed, i.e. entering the realm in which general Galois categories become co-reflective in Tannakian categories, the r\^{o}le of the ``internal'' $1$-dimensional reflector $\pi_0$ is played at the $2$-dimensional level by the (co-) reflective association $$\tt\mapsto \Pi(\tt).$$ In this case, the band of the Tannakian gerbe of $\Pi(\tt)$ corresponds to a classical profinite (Stone) group, and the gerbe becomes a gerbe with a Stone band, which in turn corresponds to a Galois category $\mathbf{Gal}_{\tt}$. The r\^{o}le of the external $1$-dimensional reflector $\pi_0(k)$ is played, at the $2$-dimensional level, by Carboni's construction: the Galois category $\mathbf{Gal}_{\tt}$ can be obtained from the Tannakian category $\tt$ as the Boolean pretopos that Carboni proves to be associable to any compact closed additive category with equalizer, by taking \emph{commutative separable algebras} in $\tt$.

\part{}
In Part II that follows I elaborate further on the idea discussed above, by demonstrating how the (co-) reflection works, in more detail, in the case where neutralized versions of Tannakian and Galois categories are considered. Notice that this is morally unacceptable, as interesting Tannakian categories, like those of motives, come with many available ``realization'' fibre functors, and choosing one is criminal; however doing so, as I will explain below, essentially eliminates the effects of $2$-dimensional category theory, allowing me to pass the general idea in a simpler categorical environment. I give all the definitions and results in detail.

\section{Galois categories}
\subsection{Definitions}
The main reference for Galois categories is \cite{sga1}. The definition that follows is not Grothendieck's original axiomatic characterization, but it is equivalent to it; it consists of a more compact set of axioms than those appearing in \cite{sga1}.
\begin{definition}[Galois category]\label{definitiongal}
A \emph{Galois category} consists of a pair $(\c,F)$, where $\c$ is an (essentially) small category which is 
\begin{itemize}
\item[(GAL1)] finitely complete,
\item[(GAL2)] finitely cocomplete,
\item[(GAL3)] and every object is a finite coproduct of connected objects.
\end{itemize}
for which there exists a functor
\[
  \begin{tikzpicture}
  \node (A) at (0,0) {$\c$};
  \node (B) at (0,-1.5) {$\set_f$};
 \draw[->] (A) to node[label=left:$\scriptstyle p$]{} (B);
\end{tikzpicture}
\]
to the category $\set_f$ of finite sets, called a \emph{fibre functor on $\c$}, which
\begin{itemize}
\item[(GAL4)] is exact, and
\item[(GAL5)] is conservative (i.e. reflects isomorphisms: if $f$ is a morphism in $\c$ such that $p(f)$ is an isomorphism, then $f$ is an isomorphism).
\end{itemize}
I will say \emph{$\c$ is a Galois category} if I do not make a particular choice of a fibre functor on it\footnote{Notice however, that any two fibre functors on a Galois category $\c$ are (non-canonically) isomorphic.}, while a pair $(\c,p)$ consisting of a Galois category $\c$ and a chosen fibre functor $p$ will be called a \emph{neutralized Galois category}. 
\end{definition}
\begin{remark}
The definition of a Galois category is equivalent to requiring a Boolean pretopos admitting a point (consequently a Boolean topos). See \cite{johnstone} for this formulation of the theory, a perspective on Galois categories which will be necessary later when discussing Carboni's work.
\end{remark}
In this work I will be studying (neutralized) Galois categories collectively so let's see what a $1$-morphism between them should be. 
\begin{definition}[$1$-morphisms of Galois categories]
Let $\c$ and $\c'$ be Galois categories. 
\begin{itemize}
\item A \emph{$1$-morphism of Galois categories from $\c$ to $\c'$} is a functor $F:\c\to\c'$ which is exact.
\item Given fibre functors $p, p'$ neutralizing $\c$ and $\c'$ respectively, a \emph{$1$-morphism of neutralized Galois categories $(\c,p)\to(\c',p')$} is a pair $(F,\alpha)$ consisting of a morphism $F:\c\to\c'$ of Galois categories and a chosen natural isomorphism $\alpha:p\cong p'\circ F$ as follows
\[
\begin{tikzpicture}[scale=2]
\node (S) at (0,0) {$\set_f$};
\node (C) at (-1,1) {$\c$};
\node (C') at (1,1) {$\c'$};

\draw[->] (C) to node (f) {} node[pos=0.5,label=left:$\scriptstyle p$] {} (S);
\draw[->] (C) to node (g) {} node[pos=0.5,label=above:$\scriptstyle F$] {} (C');
\draw[->] (C') to node (f') {} node[pos=0.5,label=right:$\scriptstyle p'$] {} (S);
\draw[shorten >=0.5cm,shorten <=0.7cm,double,double distance = 0.05cm,->] (C') -- node[above]{$\scriptstyle \alpha$} node[below]{$\scriptstyle \cong$} (f);
\end{tikzpicture}
\] 
\end{itemize}
\end{definition}
Next let's see what $2$-morphisms are.
\begin{definition}[$2$-Morphisms of neutralized Galois categories]
Let $(\c,p)$ and $(\c',p')$ be neutralized Galois categories. A \emph{$2$-morphism of neutralized Galois categories} from $(F,\alpha)$ to $(F',\alpha')$ is a natural transformation $\tau:F\Rightarrow F'$ such that we have an equality $$\alpha'\cdot \tau=\alpha$$ of natural transformations, where $\alpha'\cdot \tau$ is the vertical composition $$\alpha'\cdot \tau:=\alpha'\cdot(\iota_{p'}\underset{2}{\circ} \tau):p'\circ F\Rightarrow p:\c\to \set_f,$$ where $\iota_{p'}\underset{2}{\circ} \tau$ is the whiskering 
\[
\begin{tikzpicture}[scale=1.5]
\node (A) at (0,0) {$\c$};
\node (B) at (2,0) {$\c'$};
\node (C) at (4,0) {$\set_f$};
\node (A') at (5,0) {$\c$};
\node (C') at (7,0) {$\set_f$};

\draw[->] (A) to[bend left=50] node (f) {} node[pos=0.7,label=above:$\scriptstyle F$] {} (B);
\draw[->] (A) to[bend right=50] node (g) {} node[pos=0.7,label=below:$\scriptstyle F'$] {} (B);
\draw[->] (B) to[bend left=50] node (f') {} node[pos=0.7,label=above:$\scriptstyle p'$] {} (C);
\draw[->] (B) to[bend right=50] node (g') {} node[pos=0.7,label=below:$\scriptstyle p'$] {} (C);
\draw[->] (A') to[bend left=50] node (f'') {} node[pos=0.7,label=above:$\scriptstyle p'\circ F$] {} (C');
\draw[->] (A') to[bend right=50] node (g'') {} node[pos=0.7,label=below:$\scriptstyle p'\circ F'$] {} (C');

\draw[->,decorate,decoration={snake,post length=0.9mm}] (C) to (A');

\draw[double,double distance = 0.05cm,->] (f) -- node[right]{$\scriptstyle \tau$} (g);
\draw[double,double distance = 0.05cm,->] (f') -- node[right]{$\scriptstyle  \iota_{B'}$} (g');
\draw[double,double distance = 0.05cm,->] (f'') -- node[right]{$\scriptstyle \iota_{B'}\underset{2}{\circ} \tau$} (g'');
\end{tikzpicture}
\]
with $\underset{2}{\circ}$ denoting the horizontal composition operation of natural transformations and $\iota_{p'}$ is the identity natural transformation on the functor $p'$.
\end{definition}
\begin{proposition}[The $2$-category of neutralized Galois categories]
Neutralized Galois categories, $1$-morphisms between them and $2$-morphisms between those, as defined above, assemble into a $2$-category denoted by $\gal_{*}$. 
\end{proposition}
We must immediately notice that this $2$-category is not exactly a $2$-category: it is, rather a ``setoid-enriched category'', i.e. $\gal_{*}$ is naturally $2$-equivalent to a $1$-category: one can check that $2$-morphisms are isomorphisms and automorphisms of $1$-morphisms are the identity. Indeed, notice that on components, for every $C\in \ob\c$, we have by the Godement rule that 
\begin{eqnarray}\label{eqq}
\alpha_C=(\alpha'\cdot \tau)_C=\alpha'_C\circ(\iota_{p'}\underset{2}{\circ} \tau)_C=\alpha'_C\circ \id_{pF'C}\circ p'(\tau_C)=\alpha'_C\circ p'(\tau_C),
\end{eqnarray}
implying that, since both $\alpha_C$ and $\alpha'_C$ are isomorphisms, $p'(\tau_C)$ must be an isomorphism too. But it then immediately follows from requirement $\text{(GAL5)}$ that $\tau_C$ is also an isomorphism for every $C\in \ob\c$, so $2$-morphisms are indeed isomorphisms and we are half-way there. We need to also show that automorphisms of $1$-morphisms are the identity $2$-morphism. But this follows easily by noticing that $(\ref{eqq})$ becomes in this case $$\alpha_C=\alpha_C\circ p'(\tau_C),$$ which in turn implies $$p'(\tau_C)=\id_{p'C}$$ since $\alpha_C$ is an isomorphism. It then suffices to notice that a fibre functor is faithful and the identity already gives $p'(\id_C)=\id_{p'C}$. Thus $\tau_C$ must indeed be the identity itself.
\begin{remark}
In this way $\gal_{*}$ becomes equivalent to a $1$-category, and I will be treating it as such in the sequel - hence the prefix ``$2$-'' in parentheses in front of every instance of referring to neutralized categories collectively. Notice that this would by far \emph{not} be the case if we were assembling general Galois categories into a $2$-category, where we do not have a ``canonical'' choice of fibre functor: it is precisely the \emph{fixing} of a fibre functor that allows us to treat these categories collectively as a $1$-category. 
\end{remark}
\subsection{Classification of neutralized Galois categories}
For a profinite group $\pi$, denote by $$\cont_f(\pi)$$ the category of discrete finite sets endowed with a continuous (left) action of $\pi$. It is a Galois category admitting the forgetful functor 
\[
\begin{tikzpicture}
\node (A) at (0,0) {$\set_f$};
\node (B) at (0,1.5) {$\cont_f(\pi)$};

\draw[->] (B) to node (f) {} node[label=left:$\scriptstyle U$] {} (A);
\end{tikzpicture}
\]
to the category of finite sets as a fibre functor. It turns out that all (neutralized) Galois categories are of this form. To see this we need the following definition.
\begin{definition}[The based fundamental group of a Galois category]
Let $\c$ be a Galois category and let $p$ be a fibre functor. The \emph{(Galois) fundamental group of $\c$ pointed at $p$} is the group $$\pi_1^{\ga}(\c,p):=\aut(p)$$ of automorphisms of the fibre functor $p$.
\end{definition}
This fundamental group is a profinite Stone group when topologized as a closed subgroup $$\pi_1^{\ga}(\c,p)\subseteq\prod_{C\in\ob\c} \aut(FC),$$ of the product group above, which is itself a profinite group when endowed with the product topology (each of the finite groups $\aut(FC)$ is given the discrete topology), as an arbitrary product of profinite groups is profinite and discrete finite groups are profinite. 
\par Grothendieck's work on Galois categories has the following two equally important aspects:
\begin{itemize}
\item \textbf{Reconstruction:} For any profinite group we have $\pi\cong\pi_1^{\ga}(\cont_f(\pi),U)$.
\item \textbf{Recognition:} A neutralized Galois category $(\c,p)$ factorizes through an equivalence
\[\begin{tikzcd}[row sep=large, column sep=large]
\c \arrow[r, "\simeq"]\arrow[d, "p"']& \cont_f\left(\pi_1^{\ga}(\c,p)\right) \arrow[ld, "U"]\\
\set_f
\end{tikzcd}
\]
\end{itemize}
These constructions are ($2$-) functorial and we have the following result.
\begin{theorem}[Grothendieck Characterization of neutralized Galois categories]\label{thmgal}
The assignment $$(\c,p)\mapsto \pi_1^{\ga}(\c,p)$$ induces a ($2$-) functor $$\Aut: \gal_{*}\to \stonegrp^{\op}$$ which is an equivalence of ($2$-) categories (when $\stonegrp$ is endowed with identities as $2$-morphisms). The inverse to $\Aut$ is $\Cont:\stonegrp^{\op}\to\gal_{*}$, which is the ($2$-) functor that assigns to a profinite group $\pi$ the neutralized Galois category $(\cont_f(\pi), U)$.
\end{theorem}
\begin{proof}
See \cite{sga1}, Expos\'{e} V.
\end{proof}
\begin{remark}
If we did not consider Galois categories with fixed fibre functors on them, a fully $2$-categorical equivalence statement is needed, in which Galois categories should be compared with Stone gerbes.
\end{remark}

\section{Carboni's separable algebras}
Every field $k$ defines its Galois category $\fet_k$, the category of ``finite \'{e}tale covers over $\spec k$'', and to get a fibre functor on it one needs to pick a separable closure of $k$ (which is where the virtue of choosing $k$ to be separably closed in this article comes in). This geometrically defined category admits a description using the algebraic notion of ``separable'' or ``\'{e}tale'' $k$-algebra, which is precisely the notion on which Carboni builds his generalized notion of \emph{commutative separable algebra in a symmetric monoidal category} that I will be using to pass from Tannakian to Galois categories. In order to give meaning to the statements I made above, I would like to discuss the notion of finite \'{e}tale schemes in some detail and in particular explain their algebraic side. Thus this section starts with a brief introduction to the relevant scheme-theoretic notions.
\subsection{Basics of Scheme Theory}
Let's start with defining some fundamental scheme-theoretic notions.
\begin{itemize}
\item  A \emph{ringed space} is a pair $(V,\mathcal{O}_V)$ consisting of a topological space $V$ and a sheaf of rings $\mathcal{O}_V$ called the \emph{structure sheaf}. 
\item A \emph{locally ringed space} is a ringed space $(V,\mathcal{O}_V)$ if the stalks $\mathcal{O}_{V,v}$ of $\mathcal{O}_V$ at every $v\in V$ are local rings. A \emph{morphism of locally ringed spaces} $(V,\mathcal{O}_V)\to(W,\mathcal{O}_W)$ is a pair $(f,f^{\sharp})$ consisting of 
\begin{itemize}
\item a continuous morphism $f:V\to W$ of topological spaces,
\item a morphism $f^{\sharp}:\mathcal{O}_W\to f_{\ast}\mathcal{O}_V$ of sheaves of rings,
\item such that for any $v\in V$ the homomorphism of rings on the stalks $$f^{\sharp}_v:\mathcal{O}_{W,f(v)}\to \mathcal{O}_{V,v}$$ is a homomorphism of local rings. 
\end{itemize}
\item For any commutative ring $R$ one can define a topology on the set $\spec R$ of prime ideals of $R$, called the \emph{Zariski topology}, having a basis of open sets given, for every $r\in R$, by the sets $$D(r):=\lbrace P\;:\;\text{$P$ prime with $r\notin P$} \rbrace,$$ and which admits a structure sheaf $\mathcal{O}_{\spec R}$, making $(\spec R, \mathcal{O}_{\spec R})$ into a locally ringed space called the \emph{(prime) spectrum of $R$}. Notice that spectra of fields are just points as the only prime ideal of a field $k$ is the trivial ideal so $\spec k=\lbrace (0)\rbrace$. 
\item An \emph{affine scheme} is a locally ringed space $(V,\mathcal{O}_V)$ isomorphic to the ringed space $(\spec R, \mathcal{O}_{\spec R})$ for some commutative ring $R$. A \emph{morphism of affine schemes} is a morphism of locally ringed spaces. The category of affine schemes will be denoted by $\aff$. 
\item A \emph{scheme} is a locally ringed space that admits an open covering by affine schemes. A \emph{morphism of schemes} is a morphism of locally ringed spaces. The category of schemes will be denoted by $\sch$. 
\end{itemize}
In fact the construction of the (affine) scheme $\spec R$ out of a commutative ring $R$, assembles into a fully faithful functor
\begin{eqnarray*}
\spec:\cring^{\op}&\to &\sch\\
R&\mapsto & (\spec R,\mathcal{O}_{\spec R}).
\end{eqnarray*}
Taking global sections of the structure sheaves of schemes (or \emph{taking the coordinate ring $\mathcal{O}(V)$ of a scheme $(V,\mathcal{O}_V)$}, terminology inherited from the situation generalized here from the level of affine varieties) induces a functor 
\begin{eqnarray*}
\mathcal{O}:\sch&\to &\cring^{\op}\\
(V,\mathcal{O}_V) &\mapsto & \Gamma_V(\mathcal{O}_V)=\mathcal{O}_V(V)=:\mathcal{O}(V),
\end{eqnarray*}
which is left adjoint to the functor $\spec$. In this way $\cring^{\op}$ becomes a reflective subcategory of $\sch$, 
\[\begin{tikzcd}[row sep=large, column sep=large]
\sch \arrow[d, shift left=1ex, "\mathcal{O}"]\\
\cring^{\op} \arrow[u, hook, shift left=1ex, "\spec", "\vdash"']
\end{tikzcd}
\]
which realizes $\cring^{\op}$ as $$\cring^{\op}\simeq \aff,$$ the full subcategory of $\sch$ consisting of affine schemes. 
\begin{remark}
We will refer to the adjunction above as \emph{the structure adjunction}.
\end{remark}
Now let us recall \emph{Grothendieck's relative point of view}, in which the focus is not just on schemes in isolation, but rather, schemes relative to some ``base scheme''.
\begin{definition}
Let $S=(V,\mathcal{O}_V)$ be a scheme. 
\begin{itemize}
\item A \emph{scheme over $S$}, or an \emph{$S$-scheme} is an object of the over-category $\sch/S$, i.e. a pair $(X,f)$, consisting of a scheme $X$ together with a morphism of schemes $f:X\to S$, called the \emph{structural morphism}.
\item If $(Y,g)$ is another $S$-scheme, a \emph{morphism of $S$-schemes} from $(X,f)$ to $(Y,g)$ is a morphism in $\sch/S$, i.e. it is a morphism of schemes $u:X\to Y$,
\[\begin{tikzcd}[row sep=large, column sep=large]
X\arrow[rr, "u"]\arrow[rd,"f"]&&Y\arrow[ld, "g"]\\
&S,
\end{tikzcd}
\] 
such that $f=g\circ u$.
\end{itemize} 
\end{definition}
\begin{remark}
If the scheme $S$ is the spectrum $\spec R$ for some ring $R$, we will refer to a scheme over $S$ as a \emph{scheme over $R$}, or an \emph{$R$-scheme}. 
\end{remark}
Every ring admits a unique morphism from the ring of integers $\mathbb{Z}$, and every scheme admits a unique morphism to the scheme $\spec \mathbb{Z}$, courtesy of the structure adjunction bijection $$\hom_{\sch}(X,\spec R)\cong\hom_{\cring}(R, \mathcal{O}(X)),$$ allowing us to think of all schemes as \emph{schemes over $\mathbb{Z}$}. We want to obtain a ``relative to a base scheme $S$'' version of the structure adjunction. For this, we need to be able to ``base-change''.
\begin{proposition}
The category $\sch$ of schemes has pullbacks. 
\end{proposition} 
\begin{proof}
This is essentially a consequence of the tensor product of rings. For affine schemes $X=\spec A$, $Y=\spec B$ and $Z=\spec C$, for rings $A, B$ and $C$, the pullback scheme 
\[\begin{tikzcd}[row sep=large, column sep=large]
X\times_Y Z\arrow[d, dotted] \arrow[r, dotted]&X=\spec A\arrow[d]\\
Z=\spec C\arrow[r]&Y=\spec B
\end{tikzcd}
\] 
is the affine scheme corresponding to the tensor product of rings, $$X\times_Y Z=\spec (A\otimes_B C),$$ and more generally, the pullback of schemes is defined by glueing together affine pullbacks of the above form.
\end{proof}
Now let us fix a ring $R$. Since $\sch$ has pullbacks, we can look at the ``over $R$'' version of the structure adjunction.
\begin{proposition}
For any ring $R$, the structure adjunction gives an adjunction
\[\begin{tikzcd}[row sep=large, column sep=large]
\sch/\spec R \arrow[d, shift left=1ex, "\mathcal{O}_R"]\\
\cring^{\op}/R, \arrow[u, hook, shift left=1ex, "\spec_R", "\vdash"']
\end{tikzcd}
\]
where $\mathcal{O}_R$ is the evident induced functor and $\spec_R$ is the composite
\[\begin{tikzcd}[row sep=large, column sep=large]
\cring^{\op}/R\cong\cring^{\op}/(\mathcal{O}\circ \spec R)\to \sch/(\mathcal{O}\circ\spec R)\to \sch/\spec R
\end{tikzcd}
\] 
of the evident functor induced by $\spec$ with the pullback along the $(\mathcal{O}\dashv \spec)$-unit at $\spec R$.
\end{proposition}
\begin{proof}
Standard category theory - notice that the natural isomorphism $R\cong \Gamma\circ \spec R$ for any $R\in\cring$ is a consequence of the fact that $\cring^{\op}$ is reflective in $\sch$.
\end{proof}
\begin{remark}[Affine $R$-schemes]
The slice category $\cring^{\op}/R$ can be identified with the category $\calg_R^{\op}$, opposite to the category of commutative $R$-algebras under the relative structure adjunction, and we have $$\calg_R^{\op}\simeq \aff/\spec R=:\aff_R,$$
with the objects of the category $\aff_R$ often called \emph{affine $R$-schemes}.
\end{remark}
\subsection{Finite \'{e}tale covers and classical separable $k$-algebras}
The motivating example of a Galois category for Grothendieck was the category $\fet_X$ of finite \'{e}tale covers of a connected scheme $X$, i.e. the full subcategory $$\fet_X\subset \sch/X$$ consisiting of the schemes over $X$ which are ``finite and \'{e}tale\footnote{I will not go into the definitions of these morphisms in detail as that would divert me from the main goal of this paper,  especially because in the case of interest for us, namely when $X=\spec k$, finite \'{e}tale $X$-schemes take a much simpler form - see the next remark.}''. For the general definitions see Expos\'{e} I of \cite{sga1}.
\begin{remark}\label{finitenesset}
A morphism $S\to\spec k$ is finite and \'{e}tale if and only if $S$ is a finite disjoint union of spectra of finite separable field extensions of $k$. In particular, for a separably closed field $\Omega$, we have $$\fet_{\spec\Omega}\simeq \set_f.$$
\end{remark}
\par For every choice of a geometric point $\overline{x}:\spec\Omega\to X$ of $X$, where $\Omega$ is a separably closed field, and for any finite \'{e}tale $X$-scheme $f:S\to X$, consider the pullback 
\[
\begin{tikzpicture}
\node (A) at (0,0) {$S\times_X \spec\Omega$};
\node (B) at (2,0) {$S$};
\node (C) at (2,-2) {$X$};
\node (D) at (0,-2) {$\spec \Omega$};
\draw[dashed, ->] (A) to node (f) {} node[label=above:$\scriptstyle $] {} (B);
\draw[->] (B) to node (f) {} node[label=right:$\scriptstyle f$] {} (C);
\draw[dashed,->] (A) to node (f) {} node[label=left:$\scriptstyle $] {} (D);
\draw[->] (D) to node (f) {} node[label=below:$\scriptstyle \overline{x}$] {} (C);
\end{tikzpicture}
\]
and denote by $$\phi_{\overline{x}}(S):=|S\times_X \spec\Omega|$$ the underlying set of the finite \'{e}tale scheme $S\times_X \spec\Omega$ over $\spec \Omega$, which is finite by Remark \ref{finitenesset}. In this way we get a fibre functor 
\[
\begin{tikzpicture}
\node (A) at (0,0) {$\scriptstyle \set_f$};
\node (B) at (0,1.5) {$\scriptstyle \fet_X$};
\draw[->] (B) to node (f) {} node[label=left:$\scriptstyle \phi$] {} (A);
\end{tikzpicture}
\]
on $\fet_X$ sending a finite \'{e}tale cover $S$ of $X$ to $\phi_{\overline{x}}(S)$. 
\begin{remark}
The consideration of $\fet_X$ allowed Grothendieck to define a natural generalization of the absolute Galois group of a field $k$, the \emph{\'{e}tale fundamental group of $X$ based at a geometric point $\overline{x}$}: the Galois fundamental group of $\fet_X$ based at $\phi$. It provides a parallel, in Algebraic Geometry, to the (Poincar\'{e}) fundamental group of topological spaces based at a point.
\end{remark}
Let's now discuss the case $X=\spec k$ in more detail. In this case, a finite \'{e}tale scheme $S\to\spec k$ takes the form $$S=\spec A,$$ where $A$ is a ``(finite) commutative separable $k$-algebra''. Let's recall the definition of these  objects, as it is on this algebraic side of the theory that Carboni builds his monoidal notion.
\begin{definition}[Classical Commutative Separable algebras over a field] 
Let $k$ be a field. A commutative algebra $A$ over $k$ is \emph{classically separable} if for every field extension $k\subseteq k'$ of $k$, $J(A\otimes_k k')=0$, i.e. $A\otimes_k k'$ has zero Jacobson radical. 
\end{definition}
I will denote by $$\fsep_k$$ the full subcategory of $\calg_k$ on the finite\footnote{Recall that a commutative $k$-algebra $A$ is \emph{finite over $k$} if it is finite-dimensional as a $k$-vector space.} separable $k$-algebras. The structure adjunction induces an equivalence \begin{eqnarray}\label{fetsep}
\fet_{\spec k}\simeq \fsep_k^{\op}.
\end{eqnarray}
Notice the following standard characterization of separability, which is the algebraic parallel of Remark \ref{finitenesset}.
\begin{proposition}[Finite Commutative Separable algebras over a field]\label{propsep}
Let $A$ be a finite commutative algebra over $k$. Then $A$ is classically separable if and only if it is a finite direct product of finite separable field extensions of $k$.
\end{proposition}
Thus when $k$ is separably closed, every finite commutative separable algebra over $k$ is of the form $k^n$, inducing the expected contravariant equivalence $$\fsep_k^{\op}\simeq \set_f.$$
Under equivalence $(\ref{fetsep})$, the fibre functor $p:\fsep_k^{\op}\to\set_f$ can be identified with the functor
\[
\begin{tikzpicture}
\node (A) at (0,0) {$\set_f$};
\node (B) at (0,1.5) {$\fsep_k^{\op}$};
\draw[->] (B) to node (f) {} node[label=left:$\scriptstyle \text{$p=$}$, label=right:$\scriptstyle \text{$\calg_k\left(-,\overline{k}\right)$}$] {} (A);
\end{tikzpicture}
\]
\begin{remark}
Notice that $p$ is \emph{not} a representable functor: $\overline{k}$ is not a separable $k$-algebra. 
\end{remark}
The Galois fundamental group of this neutralized Galois category based at $p$ identifies as $$\pi_1^{\ga}(\fsep_k^{\op},p)\cong \gal(\overline{k}/k),$$ the absolute Galois group of $k$ (which is profinite with the Krull topology). This fibre functor thus induces, by Theorem $\ref{thmgal}$, an equivalence 
\begin{eqnarray}\label{contgal}
\fsep_k^{\op}\simeq \cont_f(\gal(\overline{k}/k))
\end{eqnarray}
for any choice $\overline{k}$ of a separable closure of $k$.
\subsection{Carboni's Commutative Separable Algebras in a monoidal category}
Let me give the definition appearing in \cite{carboni} right away; a discussion of how to recover the classical separable algebras appearing in Grothendieck's reformulation of Galois theory will follow. 
\begin{definition}[Commutative Separable Algebras in a Symmetric Monoidal Category]
A \emph{commutative separable monoid} in a symmetric monoidal category $\vv$ is an object $A\in\ob\vv$ equipped with a commutative monoid structure $(A, \mu, \eta)$ and a cocommutative comonoid structure $(A,\delta,\epsilon)$ satisfying the axioms
\begin{itemize}
\item[(U)] $\mu\circ\delta=1$, and
\item[(D)]  $(1\otimes \delta)(\mu\otimes 1)=\delta\circ\mu=(\delta\otimes 1)(1\otimes \mu)$
\end{itemize}
\begin{remark}
Notice that in \cite{carboni} the order of composition of $\mu$ with $\delta$ in his axioms (U) and (D) is accidentally reversed - the compositions should be as in the definition above.
\end{remark}
\end{definition}
This definition is inspired by the following definition, that Carboni attributes to DeMeyer and Ingraham citing \cite{demeyeringraham}. 
\begin{definition}[DeMI-separable algebras over a commutative ring $R$]\label{defcsepabs}
Let $R$ be a commutative ring and let $A$ be a not necessarily commutative algebra over $R$. We say that \emph{$A$ is a DeMI-separable algebra over $R$} if $A$ is projective as a left $A\otimes_R A^{\op}$-module under the action 
\begin{eqnarray}\label{demistr}
\cdot:(A\otimes_R A^{\op})\times A&\to& A\\
(a\otimes a',a'')&\mapsto& (a\otimes a')\cdot a'':=aa''a',
\end{eqnarray}
where $A^{\op}$ denotes the opposite ring of $A$. 
\end{definition}
One can then use the fact that for any DeMI-separable algebra $A$ over $R$ there exists a left $A\otimes_R A^{op}$-module (epi-) morphism of left $A\otimes A^{\op}$-modules
\begin{eqnarray}
m:A\otimes_R A^{\op}&\to& A\\
\sum_i a_i\otimes a_i'&\mapsto&\sum_i a_i a_i',
\end{eqnarray} 
called the {augmentation map of $A$}, for which the short exact sequence
\[
\begin{tikzpicture}
\node (A) at (0,0) {$0$};
\node (B) at (1.5,0) {$\kernel m$};
\node (C) at (3.5,0) {$A\otimes_R A^{\op}$};
\node (D) at (5.5,0) {$A$};
\node (E) at (6.7,0) {$0$};

\draw[->] (A) to (B);
\draw[->] (B) to (C);
\draw[->] (C) to (D);
\draw[->] (C) to node (g) {} node[label=above:$\scriptstyle m$] {} (D);

\draw[->] (D) to (E);
\end{tikzpicture}
\] 
splits (i.e. the augmentation map $m$ admits a section), to verify that \emph{commutative} DeMI-separable algebras over $R$ can be identified, in the symmetric monoidal category $\Proj_R$ of projective modules over a commutative ring $R$, with those commutative algebras $(A,\mu,\eta)$ admitting also a cocommutative coalgebra structure $(A,\delta,\epsilon)$, which satisfy precisely axioms (U) and (D) of Definition $\ref{defcsepabs}$.
\begin{remark}[The connection with classically separable algebras]\label{classep}
In order for an algebra $A$ over $R$ to be DeMI-separable it is required to be projective as a left module over $A\otimes_R A^{\op}$. If $A$ is also projective as a left module over $R$ (with the canonical action), then $A$ is finitely generated as an $R$-module (see \cite{demeyeringraham}, Proposition II.2.1). Thus when $R$ is a field $R=k$, the notion of commutative DeMI-separable algebra over $k$ coincides with that of a finite (classically) separable algebra over $k$ (recall that in the case of classically separable algebras commutativity is part of the definition as it is customary in Algebraic Geometry).
\end{remark}
For the purposes of this article, Carboni's terminology is quite satisfactory, as his notion of separabe algebra is used in a Galois-theoretic context. However, this notion appears in modern literature in other contexts, the following being two of them which I consider useful to keep in mind.
\begin{itemize}
\item Denoting by $\ff$ the PROP for commutative monoids and by $\ff^{\op}$ the PROP for commutative comonoids,  one can view commutative separable algebras in $\vv$ as the algebras for the composite prop $$\ff^{\op}\otimes_{\mathbb{P}}\ff$$ as in Example 5.4 of \cite{lack}.
\item Carboni's notion of a commutative separable monoid also coincides with the notion of a \emph{special commutative Frobenius monoid} appearing in more modern literature. In my opinion saying ``special'' to refer to a commutative Frobenius monoid satisfying axiom (U) of Definition $\ref{defcsepabs}$ is not ideal - perhaps ``split commutative Frobenius monoid'' is better terminology, because of the splitting of the augmentation-induced short exact sequence at the level of DeMI-separable algebras. 
\end{itemize}
I still haven't explained how to assemble commutative separable monoids in $\vv$ into a category and this is on purpose - it requires a bit of attention. Indeed one might expect that the morphisms between commutative separable monoids should be those which are \emph{both monoid and comonoid homomorphisms}. However this is not a good choice: in this case such morphisms are automatically invertible and we end up with a \emph{groupoid} of them. 
\begin{remark}[$2$-dimensional $k$-TQFTs]
The aforementioned groupoid resulting from the case $\vv=\Vect_k^{\fd}$, i.e. the groupoid $\csep(\Vect_k^{\fd})_{\mon,\comon}$ of finite separable $k$-algebras with morphisms those which are both monoid and comonoid homomorphisms, famously determines $2$-dimensional topological quantum field theories: there is an equivalence of categories $$\mathbf{2dTQFT}_k:=\mathbf{SymMon}(\mathbf{2dCob},\Vect_k^{\fd})\simeq\csep(\Vect_k^{\fd})_{\mon,\comon},$$ 
where $\mathbf{2dCob}$ denotes symmetric monoidal category of $2$-dimensional cobordisms, sending a $2$-dimensional $k$-TQFT $F:\mathbf{2dCob}\to\Vect_k$ to its value $F(\Sigma)$ at the circle $\Sigma$.  
\end{remark}
Instead, we want to use only one kind of morphisms, either monoid or comonoid homomorphisms.
\begin{definition}
Let $\vv$ be a symmetric monoidal category. We will denote by $$\csep(\vv)_{\mon}$$ the category of commutative separable monoids in $\vv$ with morphisms the monoid homomorphisms, and I will denote by $$\csep(\vv)_{\comon}$$ the category of commutative separable monoids in $\vv$ with morphisms the comonoid homomorphisms.
\end{definition}
Carboni proves that these two categories are dual to each other.
\begin{proposition}
Let $\vv$ be a symmetric monoidal category. We have an equivalence of categories $$\csep(\vv)_{\comon}\simeq \csep(\vv)_{\mon}^{\op}.$$
\end{proposition}
\begin{proof}
See the Corollary of Section 3 in \cite{carboni}.
\end{proof}
Let's define next the categories on which Carboni's construction will be applied.
\section{Tannakian Categories}
The main references for Tannakian categories are \cite{deligne}, \cite{delignemilne} and \cite{saavedra}.
\subsection{Definitions}
Throughout this section $k$ will denote a field.
\begin{definition}[Pre-Tannakian category] \label{defpretann}
A \emph{pre-Tannakian category over $k$} or \emph{$k$-pre-Tannakian category} is an essentially small category $\tt$
\begin{itemize}
\item[(PRETAN1)] which is symmetric monoidal with  tensor product $\otimes$ and tensor unit $\mathbf{1}$,
\item[(PRETAN2)] which is rigid (i.e. every object $T$ has a dual $T^{\vee}$),
\item[(PRETAN3)] and abelian \emph{as a symmetric monoidal category}, i.e. $\otimes$ is required to be biadditive, and finally
\item[(PRETAN4)] such that the endomorphism ring of the tensor unit satisfies $\End(\mathbf{1})\cong k$.
\end{itemize}
\end{definition}
\begin{remarks}
I would like to make the following remarks:
\begin{itemize}
\item As it is well-known, the requirement (PRETAN2) of rigidity on top of (PRETAN1) makes a symmetric monoidal category into what is known as a \emph{compact closed} category: such categories admit an \emph{internal hom} bifunctor
\begin{eqnarray*}
[-,-]:\tt^{\op}\otimes\tt&\to&\tt\\
(T,L)&\mapsto& T^{\vee}\otimes L,
\end{eqnarray*}
turning $\tt$ into a closed monoidal category.
\item Being abelian and thus pre-abelian, a pre-Tannakian category $\tt$ has all finite limits and all finite colimits.
\item Only assuming (PRETAN1)-(PRETAN3), notice that $\End(\mathbf{1})$ becomes a monoid in abelian groups, i.e. a ring, and it is commutative by the Eckmann-Hilton argument. The left unitor induces the structure of an $\End(\mathbf{1})$-module on every object $T\in\ob\tt$ and (PRETAN4) makes every object into a $k$-vector space. In fact, a pre-Tannakian category $\tt$ becomes $k$-linear in such a way that $\otimes$ becomes $k$-bilinear. 
\end{itemize} 
\end{remarks}
The notion of a fibre functor is more complicated in the case of Tannakian categories than in the case of Galois categories. 
\begin{definition}[Tannakian Fibre Functors] 
Let $\tt$ be a pre-Tannakian category over $k$. A \emph{fibre functor on $\tt$ with with values in a $k$-algebra $A$} is a functor $\omega:\tt\to \Mod_A^{\fg}$ to the category of finitely generated $A$-modules,
\[
  \begin{tikzpicture}
  \node (A) at (0,0) {$\tt$};
  \node (B) at (0,-1.5) {$\Mod_A^{\fg}$};
 \draw[->] (A) to node[label=left:$\scriptstyle \omega$]{} (B);
\end{tikzpicture}
\]
which is 
\begin{itemize}
\item[(FIB1)] strong monoidal,
\item[(FIB2)] $k$-linear,
\item[(FIB3)] exact, 
\item[(FIB4)] and valued in the subcategory $\Proj_A^{\fg}$ of (finitely generated) projective $A$-modules.
\end{itemize}
\end{definition}
\begin{remark}
Similarly to the case of Galois categories, it is a consequence of the axioms that a Tannakian fibre functor is faithful. 
\end{remark}
We can now give the definition of a Tannakian category.
\begin{definition}[Tannakian category]
Let $\tt$ be a pre-Tannakian category. 
\begin{itemize}
\item We say that $\tt$ is \emph{$k$-Tannakian} if it admits a fibre functor with values in some (nonzero) $k$-algebra $A$.
\item We say that a Tannakian category is \emph{neutral} if it admits a fibre functor with values in $k$, i.e. if it admits a fibre functor 
\[
  \begin{tikzpicture}
  \node (A) at (0,0) {$\tt$};
  \node (B) at (0,-1.5) {$\Vect_k^{\fg}$};
 \draw[->] (A) to node[label=left:$\scriptstyle \omega$]{} (B);
\end{tikzpicture}
\]
to the category of finite-dimensional $k$-vector spaces. 
\item Finally, a \emph{neutralized Tannakian category over $k$} is a pair $(\tt,\omega)$ consisting of a neutral $k$-Tannakian category $\tt$ and a fibre functor $\omega:\tt\to\Vect_k^{\fd}$ with values in $k$. 
\end{itemize}
\end{definition}
\begin{remarks}
Notice the following:
\begin{itemize}
\item Not every pre-Tannakian category over $k$ is Tannakian: it is possible that no fibre functor with values in some $k$-algebra $A$ exists on a pre-Tannakian $\tt$.
\item Notice also that, contrary to the case of Galois categories, if $\tt$ is a neutral Tannakian category, and $\omega_1,\omega_2:\tt\to\Vect_k^{\fd}$ are two fibre functors on it valued in $k$, we do \emph{not} necessarily have a (monoidal) isomorphism $\omega_1\cong\omega_2$.
\end{itemize}
\end{remarks}
As in the case of Galois categories, we will be studying neutralized $k$-Tannakian categories collectively so and let's define $1$-morphisms.
\begin{definition}[$1$-morphisms of Tannakian categories]
Let $\tt$ and $\tt'$ be Tannakian categories over $k$. 
\begin{itemize}
\item A \emph{$1$-morphism of Tannakian categories from $\tt$ to $\tt'$} is a functor $F:\tt\to\tt'$ which is 
\begin{itemize}
\item strong monoidal,
\item $k$-linear,
\item and exact.
\end{itemize}
\item Given neutral fibre functors $\omega, \omega'$ on $\tt$ and $\tt'$ respectively, a \emph{$1$-morphism of neutralized Tannakian categories $(\tt,\omega)\to(\tt',\omega')$} is a pair $(F,\alpha)$ consisting of a morphism $F:\tt\to\tt'$ of Tannakian categories and a chosen monoidal natural isomorphism $\alpha:\omega\cong \omega'\circ F$ as follows 

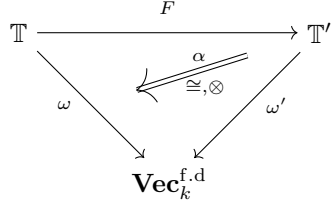
\[
\begin{tikzpicture}[scale=2]
\node (S) at (0,0) {$\Vect_k^{\fd}$};
\node (C) at (-1,1) {$\tt$};
\node (C') at (1,1) {$\tt'$};

\draw[->] (C) to node (f) {} node[pos=0.5,label=left:$\scriptstyle \omega$] {} (S);
\draw[->] (C) to node (g) {} node[pos=0.5,label=above:$\scriptstyle F$] {} (C');
\draw[->] (C') to node (f') {} node[pos=0.5,label=right:$\scriptstyle \omega'$] {} (S);

\draw[shorten >=0.5cm,shorten <=0.7cm,double,double distance = 0.05cm,->] (C') -- node[above]{$\scriptstyle \alpha$} node[below]{$\scriptstyle \;\;\cong,\otimes$} (f);
\end{tikzpicture}
\]
\end{itemize}
\end{definition}
Next let's see what $2$-morphisms are.
\begin{definition}[$2$-Morphisms of neutralized Tannakian categories]
Let $(\tt,\omega)$ and $(\tt',\omega')$ be neutralized Tannakian categories over $k$. A \emph{$2$-morphism of neutralized Tannakian categories} from $(F,\alpha)$ to $(F',\alpha')$ is a natural transformation $\tau:F\Rightarrow F'$ such that we have an equality $$\alpha'\cdot \tau=\alpha$$ of monoidal natural transformations, where $\alpha'\cdot \tau$ denotes vertical composition.
\end{definition}
\begin{proposition}[The $2$-category of neutralized Tannakian categories]
Neutralized Tannakian categories over $k$, $1$-morphisms between them and $2$-morphisms between those, as defined above, assemble into a $2$-category denoted by $\ntan_{k,*}$. 
\end{proposition}
\begin{remark}
As in the case of neutralized Galois categories, $\ntan_{k,*}$ becomes equivalent to a $1$-category. Since the proof of this is very similar to the one for Galois categories, I will omit it. Again, I need to stress that general Tannakian categories over $k$ can not be treated as a $1$-category.
\end{remark}
\subsection{Classification of neutralized Tannakian categories}
\subsubsection{Affine groups over a field $k$}
There are three equivalent ways to define the ``concrete'' objects corresponding to neutralized Tannakian categories over a field $k$: the ``geometric'', ``functorial'' and ``algebraic'' languages of \emph{affine group schemes over $k$}, \emph{affine groups over $k$} and \emph{commutative Hopf algebras over $k$}, respectively. Let's define these main actors of the theory right away.
\begin{itemize}
\item \textbf{Affine group $k$-schemes:} Notice that the category $\aff_k$ is cartesian, with the product of affine $k$-schemes $X=\spec A$ and $Y=\spec B$  given by the pullback
\[\begin{tikzcd}[row sep=large, column sep=large]
X\times_{\spec k} Y\arrow[d, dotted] \arrow[r, dotted]&X=\spec A\arrow[d]\\
Y=\spec B\arrow[r]&\spec k
\end{tikzcd}
\] 
where $$X\times_{\spec k} Y=\spec (A\otimes_k B).$$ An \emph{affine group scheme over $k$} or \emph{affine group $k$-scheme} is then a group object in the category of affine $k$-schemes and their category will be denoted by $$\affgrpsch_k:=\grp(\aff_k).$$
\item \textbf{Affine $k$-groups:} An \emph{affine group over $k$} or \emph{affine $k$-group} is a functor $G:\calg_k\to\grp$, which when composed with the forgetful functor $U:\grp\to\set$ becomes representable by a $k$-algebra denoted by $\mathcal{O}(G)$ and called \emph{the coordinate algebra of $G$}
\[
  \begin{tikzpicture}[scale=0.8]
  \node (A) at (0,3) {$\calg_k$};
  \node (B) at (3,3) {$\grp$};
  \node (C) at (3,0) {$\set$};
  \node (D) at (2,2) {$\cong$};
 \draw[->] (A) to node[label=above:$\scriptstyle G$]{} (B);
 \draw[->] (B) to node[label=right:$\scriptstyle U$]{} (C);
 \draw[->] (A) to node[label=left:$\scriptstyle \text{$\calg_k\left(\mathcal{O}(G),-\right)$}$]{} (C);
\end{tikzpicture}
\]
A \emph{morphism of affine $k$-groups $G\to H$} is a natural transformation $G\Rightarrow H$ and the resulting category of affine $k$-groups will be denoted by $$\affgrp_k:=\left[\calg_k,\grp\right]_{\text{rep}}$$
This definition is equivalent to that of an affine $k$-group scheme as a consequence of the behaviour of models (group objects in this case) in presheaf toposes, which is accessible upon embedding $\aff_k\simeq \calg_k^{\op}$ in copresheaves on $\calg_k$ via Yoneda and using the fact that this Yoneda embedding preserves finite products and thus group objects. 
\item \textbf{Commutative Hopf $k$-algebras:} Reversing the point of view in the last statement, i.e. by thinking of affine $k$-groups as group objects in representable copresheaves on $\calg_k$, notice that 
\begin{itemize}
\item the group object multiplication $$\text{mult}:\calg_k(\mathcal{O}(G),-)\times\calg_k(\mathcal{O}(G),-)\Rightarrow \calg_k(\mathcal{O}(G),-)$$ induces by Yoneda a morphism $$
\delta:\mathcal{O}(G)\to\mathcal{O}(G)\otimes_k\mathcal{O}(G)$$ of $k$-algebras, called the \emph{comultiplication} of $\mathcal{O}(G)$,
\item the group object unit $$\text{unit}:\calg_k(k,-)\Rightarrow\calg_k(\mathcal{O}(G),-)$$ induces by Yoneda a morphism $$
\epsilon:\mathcal{O}(G)\to k$$ of $k$-algebras, called the \emph{counit} of $\mathcal{O}(G)$, 
\item and the group object inverse $$\text{inv}:\calg_k(\mathcal{O}(G),-)\Rightarrow\calg_k(\mathcal{O}(G),-)$$ induces by Yoneda a morphism $$
i:\mathcal{O}(G)\to \mathcal{O}(G)$$ of $k$-algebras, called the \emph{antipode} of $\mathcal{O}(G)$.
\end{itemize}
The group object interaction requirements between $\text{mult},\text{unit}$ and $\text{inv}$ translate then precisely to the requirements that the $(\delta,\epsilon,i)$ need to satisfy to endow the commutative $k$-algebra $\mathcal{O}(G)$ with the structure of a cogroup object in the category $\calg_k$, equivalently a \emph{commutative Hopf $k$-algebra}. 
\end{itemize}
To sum up, we have the equivalent perspectives
\begin{eqnarray*}
\affgrpsch_k\simeq\affgrp_k\simeq\chopf_k^{\op}.
\end{eqnarray*}
on the theory.
\paragraph{Classification}
To present the classification of neutralized Tannakian categories I will mostly use the functorial language of affine groups over $k$.
\begin{definition}[Representations of affine $k$-groups]
\item A \emph{representation of an affine $k$-group $G:\calg_k\to\grp$ on a $k$-vector space $V$} is a natural transformation $\rho:G\to \aut_V$, where $$\aut_V:\calg_k\to\grp$$ is the functor associating to a commutative $k$-algebra $A$ the group of $A$-linear $\otimes$-automorphisms of the $A$-module $V\otimes_k A$. A representation $\rho$ of $G$ on a vector space $V$ is called \emph{finite-dimensional} if $V$ is finite-dimensional as a vector space over $k$. 
\end{definition}
I will denote by $$\Rep_k^{\fd}(G)$$ the category of finite-dimensional representations of an $G$ over $k$. It is a neutral Tannakian category over $k$ admitting the forgetful functor 
\[
\begin{tikzpicture}
\node (A) at (0,0) {$\Vect_k^{\fd}$};
\node (B) at (0,1.5) {$\Rep_k^{\fd}(G)$};
\draw[->] (B) to node (f) {} node[label=left:$\scriptstyle U$] {} (A);
\end{tikzpicture}
\]
to the category of finite-dimensional $k$-vector spaces as a fibre functor. Again, as in the case of Galois categories, it turns out that all neutralized Tannakian categories are of this form. 
\begin{definition}[The based fundamental affine $k$-group of a neutral Tannakian category]
Let $\tt$ be a neutral Tannakian category and let $\omega$ be a fibre functor. The \emph{(Tannakian) fundamental affine group of $\tt$ pointed at $\omega$} is the affine group $\pi_1^{\ta}(\tt,\omega)$ whose coordinate algebra is $$\mathcal{O}\left(\pi_1^{\ta}(\tt,\omega)\right)=\End^{\vee}(\omega):=\int^{T\in\ob\tt}\omega(T)\otimes \omega(T)^{\vee},$$ the \emph{commutative Hopf algebra of coendomorphisms of $\omega$} constructed as a ($k$-linear) coend; i.e. 
$$\pi_1^{\ta}(\tt,\omega)=\calg_k\left(\End^{\vee}(\omega),-\right).$$
\end{definition}
\begin{remark}
This definition is given in a perhaps uncommon way: usually the Tannakian fundamental affine group associated to a fibre functor $\omega$ is defined using ($\otimes$-) endomorphisms of fibre functors. But the two perspectives are equivalent, and I use this definition to highlight the coend nature of the coendomorphism coalgebra of $\omega$, which is an often neglected aspect of the theory.
\end{remark}
\par As in the theory of Galois categories, we have a \emph{reconstruction} and a \emph{recognition} aspect of the theory:
\begin{itemize}
\item \textbf{Reconstruction:} For any affine group $G$, we have $G\cong\pi_1^{\ta}(\Rep_k^{\fd}(G),U)$.
\item \textbf{Recognition:} A neutralized Tannakian category $(\tt,\omega)$ factorizes through a monoidal equivalence
\[\begin{tikzcd}[row sep=large, column sep=large]
\tt \arrow[r, "\simeq_{\otimes}"]\arrow[d, "\omega"']& \Rep_k^{\fd}\left(\pi_1^{\ta}(\tt,\omega)\right) \arrow[ld, "U"]\\
\Vect_k^{\fd}
\end{tikzcd}
\]
\end{itemize}
\begin{remark}
For any affine group $G$ we have an equivalence of neutral Tannakian categories $$\Rep_k^{\fd}(G)\simeq\comod_k^{\fd}(\mathcal{O}(G))$$ between the category of finite-dimensional representations of $G$ and the category of finite-dimensional comodules over the coordinate commutative Hopf $k$-algebra $\mathcal{O}(G)$. Then the Reconstruction result also says that every commutative Hopf algebra $H$ can be recovered as $$H\cong\End^{\vee}(U),$$ where $U:\comod_k^{\fd}(H)\to \Vect_k^{\fd}$ is the obvious forgetful functor behaving as a fibre functor on $\comod_k^{\fd}(H)$.
\end{remark}
The main result of the section is the following.
\begin{theorem}[Saavedra-Rivano-Deligne Characterization of neutralized Tannakian categories]\label{thmtan}
The assignment $$(\tt,\omega)\mapsto \End^{\vee}(\omega)$$ induces a ($2$-) functor $$\Aut^{\vee}: \ntan_{k,*}\to \chopf_k$$ which is an equivalence of ($2$-) categories (when $\chopf_k$ is endowed with identities as $2$-morphisms). The inverse to $\Aut^{\vee}$ is $\Comod:\chopf_k\to\ntan_{k,*}$, the ($2$-) functor that assigns to a commutative Hopf algebra $H$ over $k$ the neutralized Tannakian category $(\comod_k^{\fd}(H),U)$.
\end{theorem}
\begin{proof}
See \cite{deligne}.
\end{proof}
\begin{remarks} Two remarks are due:
\begin{itemize}
\item  Firstly, perhaps the reader would expect the phrasing of Theorem $\ref{thmtan}$ to be using the assignment $$(\tt,\omega)\mapsto \pi_1^{\ta}(\tt,\omega)$$ to induce a ($2$-) functor $$\Aut: \ntan_{k,*}\to \affgrp_k^{\op}.$$ Even though this formulation is of course valid, and it would provide a more pleasing comparison with Theorem $\ref{thmgal}$, I have chosen the Hopf-algebraic formulation because this perspective will make the set-up of the comparison landscape between the two theories more apparent.    
\item Secondly, if Tannakian categories were considered without a fixed fibre functor, a fully $2$-categorical equivalence statement is needed, in which Tannakian categories are compared with affine fpqc gerbes over $\spec k$.
\end{itemize}
\end{remarks}
\section{The result}
\subsection{The ($2$-) functor $\gamma:\ntan_{k,*}\to\gal_{*}$}
The most important result in \cite{carboni}, making the connection between Tannakian and Galois categories possible, is the following.
\begin{theorem}\label{carbool}
Let $\vv$ be an essentially small, rigid, additive category with coequalizers. Then $\csep_{\comon}(\vv)$ is an essentially small Boolean pretopos. 
\end{theorem}
\begin{proof}
See Theorem 1 of Section 6 of \cite{carboni}.
\end{proof}
As noted in Remark $\ref{classep}$, for a field $k$, Carboni's construction associates to the pre-Tannakian category $\Vect_k^{\fd}$ the Galois category $$\csep_{\comon}(\Vect_k)=\fsep_k^{\op}=\csep_{\mon}(\Vect_k)^{\op}$$ dual to the category of finite classically separable algebras over $k$.  
\begin{remark}
From now on I assume $k$ to be separably closed. Notice that it then follows that $\Vect_k^{\fd}$ is sent to $\set_f$.
\end{remark}
It is also proved, in Theorem 2 of Section 6 of \cite{carboni}, that Carboni's construction associates, furthermore, to any fibre functor $$\omega:\tt\to\Vect_k^{\fd}$$ on a neutral Tannakian category $\tt$ over $k$, a fibre functor $$\csep_{\comon}(\omega):\csep_{\comon}(\tt)\to \set_f$$ on the Galois category $\csep_{\comon}(\tt)$. It in fact follows that Carboni's construction interacts well with the definitions of the ($2$-) categories of neutralized Tannakian and Galois categories given earlier in this article, inducing a ($2$-) functor 
\[
\begin{tikzpicture}
\node (A) at (0,0) {$\gal_{*}$};
\node (B) at (0,1.5) {$\ntan_{k,*}$};
\draw[->] (B) to node (f) {} node[label=left:$\scriptstyle \gamma$] {} (A);
\end{tikzpicture}
\]
from the ($2$-) category $\ntan_{k,*}$ of neutralized Tannakian categories over $k$ to the ($2$-) category $\gal_{*}$ of neutralized Galois categories. 
\begin{remark}
The details of the ($2$-) functoriality of $\gamma$ can be checked (easily but tediously) using Carboni's technology of \emph{wscc structures\footnote{These structures are called \emph{hypergraph} categories in modern literature.}}, which one can slightly adjust to get not only Galoisian fibre functors but also more general morphisms of Boolean pretoposes out of Tannakian morphisms in our context. However, for the reader who is more familiar with the language of PROPs or of Frobenius monoids, I would like to mention the following two equivalent ways to construct the category $\csep_{\comon}(\tt)$ that can be used to check the ($2$-) functoriality of $\gamma$:
\begin{itemize}
\item Using the perspective of PROPs and the language of \cite{lack}, the category $\csep_{\comon} (\tt)$ can be identified with 
the category $$(\ff^{\op}\otimes_{\mathbb{P}}\ff)(\tt)_{\ff^{\op}}$$ defined using the (eso, fully faithful) factorization of the forgetful functor $U$
\[
\begin{tikzpicture}
\node (B) at (0,0) {$(\ff^{\op}\otimes_{\mathbb{P}}\ff)(\tt)_{\ff^{\op}}$};
\node (A) at (-2,1) {$(\ff^{\op}\otimes_{\mathbb{P}}\ff)(\tt)$};
\node (C) at (2,1) {$\ff^{\op}(\tt)$};

\draw[->] (A) to node (f) {} node[label=left:$\scriptstyle $] {} (B);
\draw[right hook->] (B) to node (g) {} node[label=left:$\scriptstyle $] {} (C);
\draw[->] (A) to node (h) {} node[label=above:$\scriptstyle U$] {} (C);
\end{tikzpicture}
\]
\item Using the perspective of Frobenius monoids, the reader can identify  the category $\csep_{\comon}(\tt)$ with the category $$\mathbf{SCFrob}(\mathcal{I},\tt)_{\comon}$$ of special commutative Frobenius monoidal functors from the singleton symmetric monoidal category $\mathcal{I}$ to $\tt$ and comonoidal natural transformations between them (not both comonoidal \emph{and} monoidal natural transformations, as it is customary).
\end{itemize}
\end{remark}
Next let me remind the reader of the following definitions. 
\begin{definition}[Algebraic and \'{e}tale groups over $k$]
Let $k$ be any field.
\begin{itemize}
\item An affine $k$-group $G:\calg_k\to\grp$ over $k$ is said to be \emph{algebraic} if its coordinate algebra $\mathcal{O}(G)$ is a finitely generated $k$-algebra. I will denote by $\alggrp_k$ the category of algebraic groups over $k$.
\item An algebraic $k$-group $G:\calg_k\to\grp$ over $k$ is said to be \emph{(finite) \'{e}tale} if its coordinate algebra $\mathcal{O}(G)$ is a finite commutative separable $k$-algebra. I will denote by $\etgrp_k$ the category of \'{e}tale groups over $k$.
\end{itemize}
\end{definition}
We have identifications 
\begin{eqnarray}
\alggrp_k\simeq \grp(\calg_{k,\fg}^{\op})\simeq \chopf_{k,\fg}^{\op}
\end{eqnarray}
and 
\begin{eqnarray}\label{quotet}
\etgrp_k\simeq \grp(\fsep_k^{\op}).
\end{eqnarray}
Next recall (say from Section 6.5 of \cite{waterhouse}) that any finitely generated $k$-algebra $A$ admits a maximal separable subalgebra $$\pi_0(A).$$ 
\begin{remark}[The separable $k$-cogroup structure on $\pi_0(A)$] 
When $A$ is a finitely generated commutative Hopf $k$-algebra, i.e. a finitely generated $k$-algebra endowed with a cogroup structure, then $\pi_0(A)$ becomes a \emph{finite separable $k$-cogroup}, i.e. a cogroup object in the category $\fsep_k$ because the $\pi_0$ construction preserves coproducts, i.e. $\pi_0(A\otimes_k B)\cong \pi_0(A)\otimes_k\pi_0(B)$, for all finitely generated $k$-algebras $A$ and $B$.
\end{remark}
For any algebraic $k$-group $G$, the inclusion $$\pi_0(\mathcal{O}(G))\to \mathcal{O}(G)$$ makes $\mathcal{O}(G)$ faithfully flat over $\pi_0(\mathcal{O}(G))$ and thus the corresponding $$q_{G}:G\to G_{\et}:=\calg_k(\pi_0(\mathcal{O}(G)),-)$$ is a quotient map.
\begin{definition}
Let $G$ be an affine algebraic group over $k$. The \emph{group of connected components of $G$} is the (quotient) finite \'{e}tale algebraic group $G_{\et}$ corresponding to the finitely generated commutative Hopf subalgebra $\pi_0(\mathcal{O}(G))$ of $\mathcal{O}(G)$ under $(\ref{quotet})$. 
\end{definition}
\begin{remarks}[Connected components of algebraic $k$-groups and separable $k$-algebras]
The following two remarks shed more light on how the connected components of an algebraic $k$-group $G$ can be described using separable $k$-algebras.
\begin{itemize}
\item \textbf{The \'{e}tale quotient $G_{\et}$ and connectedness of $\spec(\mathcal{O}(G))$:} It turns out the topological space $\spec(\mathcal{O}(G))$ of the underlying affine (group) $k$-scheme of $G$ is connected if and only if $G_{\et}$ is trivial, i.e. represented by $k$. In this case $G$ is said to be a \emph{connected} algebraic $k$-group.
\item \textbf{The connected-\'{e}tale sequence of an algebraic $k$-group $G$:} The kernel $$G^0:=\ker(q_{G})$$ is a normal affine subgroup of $G$ called the \emph{connected component} or \emph{identity component} of $G$ and it is always a connected algebraic $k$-group. Thus every algebraic $k$-group participates in a short exact sequence
\[
\begin{tikzpicture}
\node (A) at (-0.8,0) {$0$};
\node (B) at (1.5,0) {$G^0=\ker q_G$};
\node (C) at (3.8,0) {$G$};
\node (D) at (5.8,0) {$\pi_0(G)$};
\node (E) at (7.6,0) {$0$};

\draw[->] (A) to (B);
\draw[->] (B) to (C);
\draw[->] (C) to (D);
\draw[->] (C) to node (g) {} node[label=above:$\scriptstyle q_G$] {} (D);

\draw[->] (D) to (E);
\end{tikzpicture}
\] 
where $\pi_0(G)$ is \'{e}tale and $G^0$ is connected, called the \emph{connected-\'{e}tale sequence of $G$}.
\end{itemize}
\end{remarks}
This construction of the \'{e}tale quotient is functorial, inducing a functor $$(-)_{\et}:\alggrp_k\to\etgrp_{k}$$ whose dual I will denote by 
$$\pi_0:\chopf_{k,\fg}\to\mathbf{Cogrp}(\fsep_k)\cong\grp(\fsep_k^{\op})^{\op}.$$
\par Now going back to assuming $k$ to be separably closed, the functor $\pi_0$ induces a functor 
\[
\begin{tikzpicture}
\node (A) at (0,0) {$\grp_f^{\op}$};
\node (B) at (0,1.5) {$\chopf_{k,\fg}$};
\draw[->] (B) to node (f) {} node[label=left:$\scriptstyle \text{$\pi_0(-)_k:=$}$, label=right:$\scriptstyle \text{$\calg_k(\pi_0(-),k)$}$] {} (A);
\end{tikzpicture}
\]
landing in the dual of the category of finite groups. I am going to be interested in the $\ind$-completion of this functor. Notice that $$\ind\grp_f^{\op}\simeq \pro(\grp_f)^{\op}\simeq\stonegrp^{\op}.$$ The $\ind$-completion of $\chopf_{k,\fg}$ is the full category of commutative Hopf algebras over $k$; this is certainly well-known but as I have not found a reference for this in the literature let me sketch the proof. So firstly notice that $\calg_k$ is a locally finitely presentable category, whose essentially small full subcategory consisting of finitely presentable objects identifies with the category $\calg_{k,\fg}$ of finitely generated commutative $k$-algebras; it is then a consequence of Gabriel-Ulmer duality that $\calg_{k,\fg}$ has finite colimits and that the embedding 
\[\begin{tikzcd}[row sep=large, column sep=large]
\calg_k\\
\calg_{k,\fg}\arrow[u, hook, "\iota"]
\end{tikzcd}
\]
preserves them, as it realizes $\calg_k$ as $\calg_k\simeq \ind\calg_{k,\fg}$, the free cocompletion of $\calg_{k,\fg}$ under filtered colimits. Hence the fully faithful $$\iota^{\op}:\calg_{k,\fg}^{\op}\to \calg_k^{\op}$$ preserves finite limits and therefore group objects can be defined to give a functor 
\begin{eqnarray*}
\grp(\iota^{\op}):\grp(\calg_{k,\fg}^{\op})&\to& \grp(\calg_k^{\op})\\
G:\mathbb{L}_{\grp}\to \calg_{k,\fg}^{\op}&\mapsto& \iota^{\op}\circ G,
\end{eqnarray*}
where $\mathbb{L}_{\grp}$ denotes the Lawvere theory of groups. Next, notice that the functor $\grp(\iota^{\op})$ is fully faithful, as we clearly have
\begin{eqnarray*}
\grp(\calg_k^{\op})\left(\iota^{\op}\circ G, \iota^{\op}\circ H\right)&\cong&\int_{\ell\in \mathbb{L}_{\grp}}\calg_k^{\op}\left(\iota^{\op}(G(\ell^{\dagger})), \iota^{\op}(H(\ell^{\ddagger}))\right)\\
&\cong& \int_{\ell\in\mathbb{L}_{\grp}}\calg_{k,\fg}^{\op}(G(\ell^{\dagger}),H(\ell^{\ddagger}))\\
&\cong& \grp(\calg_{k,\fg}^{\op})(G,H),
\end{eqnarray*}
where the second isomorphism follows from the fact that $\iota^{\op}$ is fully faithful. The dual of this functor is the fully faithful functor
\[\begin{tikzcd}[row sep=large, column sep=large]
\chopf_k\\
\chopf_{k,\fg}\arrow[u, hook, "\grp(\iota^{\op})^{\op}"]
\end{tikzcd}
\]
\begin{proposition}
The functor $\grp(\iota^{\op})^{\op}$ realizes $\chopf_k$ as the $\ind$-completion of $\calg_{k,\fg}$. 
\end{proposition}
\begin{proof}
We know already that the functor is fully faithful, and since $\calg_k$ is cocomplete, the only requirement that needs to be satisfied is that the closure of $\chopf_{k,\fg}$ in $\chopf_k$ under filtered colimits is $\chopf_k$ itself. But this follows from the fact that a commutative Hopf algebra $H$ can be written as a filtered colimit (often also called ``filtered union'' because the transition maps of the colimit are injective) $$H\cong\varinjlim_{i}H_{i}$$ 
of its finitely generated Hopf subalgebras (see section 3.3 of \cite{waterhouse}).
\end{proof}
Thus the $\ind$-completion of the functor $\pi_0$ is a functor
\[
\begin{tikzpicture}
\node (A) at (0,0) {$\stonegrp^{\op}.$};
\node (B) at (0,1.5) {$\chopf_{k}$};
\draw[->] (B) to node (f) {} node[label=left:$\scriptstyle \ind \pi_0(-)_k$] {} (A);
\end{tikzpicture}
\]

Now given a neutralized Tannakian category $(\tt,\omega)$, and writing the coendomorphism  commutative Hopf algebra $$\mathcal{O}\left(\pi_1^{\ta}(\tt,\omega)\right)\cong\varinjlim_i \mathcal{O}\left(\pi^{\ta}_1(\tt,\omega)\right)_i$$ as the filtered colimit of its finitely generated Hopf subalgebras $\mathcal{O}\left(\pi^{\ta}_1(\tt,\omega)\right)_i$, we can conclude, using Carboni's theory, that the fundamental group of the Galois fibre functor constructed from $(\tt,\omega)$ using commutative separable monoids, can be realized
\begin{eqnarray}
\pi_1^{\ga}\left(\csep(\tt)_{\comon},\csep(\omega)_{\comon}\right)\cong \varprojlim_{i}\calg_k\left(\pi_0\left(\mathcal{O}(\pi_1^{\ta}(\tt,\omega)_i,k\right) \right)
\end{eqnarray}
as the cofiltered limit of the finite groups arising as the $k$-points of the maximal separable subalgebras of each $\mathcal{O}\left(\pi^{\ta}_1(\tt,\omega)\right)_i$, inducing an identification 
\[
\begin{tikzpicture}
\node (A) at (-2,1) {$\ntan_{k,*}$};
\node (B) at (2,1) {$\chopf_{k}$};
\node (C) at (-2,-1) {$\gal_{*}$};
\node (D) at (2,-1) {$\stonegrp^{\op}$};
\node (E) at (0,0) {$\cong$};

\draw[->] (A) to node (f) {} node[label=left:$\scriptstyle \gamma$] {} (C);
\draw[->] (A) to node (f) {} node[label=below:$\scriptstyle \Aut^{\vee}$, label=above:$\scriptstyle \simeq$] {} (B);
\draw[->] (B) to node (f) {} node[label=right:$\scriptstyle \ind\pi_0(-)_k$] {} (D);
\draw[->] (D) to node (f) {} node[label=above:$\scriptstyle \Cont$, label=below:$\scriptstyle \simeq$] {} (C);
\end{tikzpicture}
\]
of $\gamma$ with the composite $\Cont\circ\ind\pi_0(-)_k\circ\Aut^{\vee}$.
\subsection{The fully faithful left ($2$-) adjoint to $\gamma$}
\par Let $G$ be a finite group and let $k$ be a general field. Recall that the \emph{constant algebraic group} $G_k$ has as underlying scheme the disjoint union of copies of $\spec(k)$ indexed by the elements of $G$, i.e. $$|G_k|=\bigsqcup_{g\in G}\spec(k);$$ 
its coordinate algebra $\mathcal{O}(G_k)$ is just the set of functions $\set(G,k)$ with is usual algebra structure and thus as an affine group over $k$, $G_k$ is the functor
\begin{eqnarray*}
G_k:\calg_k&\to&\grp\\
A&\mapsto& \calg_k(\set(G,k),A).
\end{eqnarray*}
\begin{remarks}
The reader might want to notice the following:
\begin{itemize}
\item \textbf{On the ``constant'' terminology:} We have an isomorphism $$\set(G,k)=\prod_{\gamma\in G}k=:k^G$$ as $k$-algebras (hence $\set(G,k)$ is a finite separable $k$-algebra), and for every $k$-algebra $A$ whose $\spec A$ is connected, a $k$-algebra morphism from $k^G$ to $A$ is zero on exactly one copy of $k$, hence $$G_k(A)=\calg_{k}(\set(G,k),A)=G$$ i.e. the $A$-points of $\set(G,k)$ correspond precisely to elements of $G$. This explains the terminology ``constant'' algebraic group. 
\item \textbf{The Hopf algebra structure on $k^G$:} The functor $G_k$ is indeed an algebraic $k$-group: the commutative Hopf $k$-algebra structure on $k^G$ is the following:
\begin{itemize}
\item the comultiplication is given by $\delta(\phi)(g\otimes h)=\phi(gh)$, 
\item the antipode is given by $i(\phi)(g)=\phi(g^{-1})$,
\item and the counit is given by $\epsilon(\phi)=\phi(e)$,
\end{itemize}
for all $\phi\in \set(G,k)$, $g$ and $h$ in $G$ and $e\in G$ the identity. 
\item \textbf{Comparison of $k^G$ with the group algebra \text{$k[G]$}:} The reader might want to notice that the \emph{commutative} Hopf algebra $k^{G}$ constructed above from a finite group $G$ is \emph{not} the group algebra $k[G]$ (which is also a Hopf algebra over $k$ and which is \emph{cocommutative} by construction), but rather it can be identified with its dual.
\end{itemize}
\end{remarks}
\par For the rest of the section assume $k$ to be separably closed. The assignment $G\mapsto G_k$
induces a functor $$(-)_k:\grp_f\to\alggrp_k$$ which is fully faithful, corresponding to the full embedding of \'{e}tale groups in algebraic groups. I will denote its dual functor by
\[
\begin{tikzpicture}
\node (A) at (0,0) {$\grp_f^{\op}$};
\node (B) at (0,1.5) {$\chopf_{k,\fg}$};
\draw[right hook->] (A) to node (f) {} node[label=left:$\scriptstyle k^{(-)}$] {} (B);
\end{tikzpicture}
\]
It is easy to see that $k^{(-)}$ is left adjoint to $\pi_0(-)_k$, i.e. we have a coreflection
\[\begin{tikzcd}[row sep=large]
\chopf_{k,\fg}\arrow[d,shift left=1.4ex, "\pi_0(-)_k"]\\
\grp_f^{\op}\arrow[u,hook,shift left=1.4ex, "k^{(
-)}", "\dashv"']
\end{tikzcd}
\]
of $\grp_f^{\op}$ in $\chopf_{k,\fg}$; indeed, we have, naturally in $\Gamma\in \grp$ and $H\in\chopf_{k,\fg}$, the isomorphisms
\begin{eqnarray*}
\grp_f^{\op}\left(\Gamma,\pi_0(H)_k\right)&\cong& \chopf_{k,\fg}\left(k^{\Gamma},k^{\pi_0(H)_k}\right)\\
&\cong&\chopf_{k,\fg}\left(k^{\Gamma},\pi_0(H)\right)\\
&\cong&\chopf_{k,\fg}\left(k^{\Gamma},H\right),
\end{eqnarray*} 
with the second isomorphism following from the fact that for any algebraic $k$-group $G$, $G_{\et}$ is constant as we work with $k$ separably closed, and the last isomorphism following from the fact that $k^{G}$ is separable and algebra homomorphisms from a separable algebra to $H$ factor uniquely through the inclusion $\pi_0(H)\subseteq H$. 
\par Now since both $\chopf_{k,\fg}$ and $\grp_f^{\op}$ have finite colimits, the coreflection $\ind$-completes to produce a coreflection 
\[\begin{tikzcd}[row sep=large]
\chopf_{k}\arrow[d,shift left=1.4ex, "\ind\pi_0(-)_k"]\\
\stonegrp^{\op}\arrow[u,hook,shift left=1.4ex, "\ind k^{(
-)}", "\dashv"']
\end{tikzcd}
\]
\begin{remark}[Relation with Pierce spectrum]
Let me give arrive now to the coreflection $\ind k^{(-)}\dashv \ind\pi_0(-)_k$ from a perspective which might be a bit more familiar to the category theorists, and which highlights the algebraic connection of connected components with idempotents. Recall the following two constructions:
\begin{itemize}
\item The theory of Stone duality provides us with an equivalence $$\spec_{\stone}:\bool\to \stone^{\op}$$ between the category of Boolean algebras and the dual of the category of Stone spaces, which associates to a Boolean algebra $B$ the Stone space of its ultrafilters with the Stone topology, $$\spec_{\stone}:B\mapsto\ult(B).$$ I will call this functor the \emph{Stone spectrum} functor. The inverse to $\spec_{\stone}$ is the functor $$\clopen:\stone^{\op}\to\bool,$$ sending a Stone space to its Boolean algebra of clopens. 
\item Next recall that the idempotents of a commutative ring $R$ can be endowed with the structure of a Boolean algebra, under the operations
\begin{eqnarray*}
x\land y&:=&xy\\
x\lor y&:=& x+y-xy,
\end{eqnarray*}
and this construction extends to a functor $$\idemp:\cring\to\bool$$ from the category of commutative rings to the category of Boolean algebras. 
\end{itemize}
After slicing appropriately to pass from rings to algebras, and combining the two functors above, we get, for any commutative ring $R$, a composite functor  $$\spec_{\pierce}:=\spec_{\stone}\circ\idemp^{\op}:\calg_R^{\op}\to\stone,$$
which is called the \emph{Pierce spectrum} functor. 
\par When $R$ is a field $R=k$, this functor admits a fully faithful left adjoint 
\[\begin{tikzcd}[row sep=large]
\calg_k\arrow[d,shift left=1.4ex, "\spec_{\pierce}"]\\
\stone^{\op}\arrow[u,hook,shift left=1.4ex, "\text{$\cont(-,k_{\disc})$}", "\dashv"']
\end{tikzcd}
\]
where the functor
\begin{eqnarray*}
\cont(-,k_{\disc}):\stone^{\op}&\to& \calg_k\\
X&\mapsto& \cont(X,k_{\disc})
\end{eqnarray*}
assigns to a Stone topological space $X$ the $k$-algebra of continuous functions from $X$ to $k$ endowed with the discrete topology. 
\par The Pierce spectrum functor can be checked to preserve finite coproducts, and thus we can take cogroup objects in the adjunction $\cont(-,k_{\disc})\dashv\spec_{\stone}$ to arrive to a coreflection
\[\begin{tikzcd}[row sep=large]
\cogrp(\calg_k)\arrow[d,shift left=1.4ex, "\cogrp(\spec_{\pierce})"]\\
\cogrp(\stone^{\op})\arrow[u,hook,shift left=1.4ex, "\text{$\cogrp(\cont(-,k_{\disc}))$}", "\dashv"']
\end{tikzcd}
\] 
which identifies with the coreflection $\ind k^{(-)}\dashv \ind\pi_0(-)_k$.
\end{remark}
\par We can thus define the fully faithful ($2$-) adjoint of $\gamma$ 
\[
\begin{tikzpicture}
\node (A) at (-2,1) {$\ntan_{k,*}$};
\node (B) at (2,1) {$\chopf_{k}$};
\node (C) at (-2,-1) {$\gal_{*}$};
\node (D) at (2,-1) {$\stonegrp^{\op}$};
\node (E) at (0,0) {$:=$};

\draw[right hook->] (C) to node (f) {} node[label=left:$\scriptstyle \text{$(-)_{\lin}$}$] {} (A);
\draw[->] (B) to node (f) {} node[label=below:$\scriptstyle \Comod$, label=above:$\scriptstyle \simeq$] {} (A);
\draw[right hook->] (D) to node (f) {} node[label=right:$\scriptstyle \text{$\ind k^{(-)}$}$] {} (B);
\draw[->] (C) to node (f) {} node[label=above:$\scriptstyle \Aut$, label=below:$\scriptstyle \simeq$] {} (D);
\end{tikzpicture}
\]
as the composite $(-)_{\lin}:=\Comod\circ\ind k^{(-)}\circ\Aut$, inducing a $2$-coreflection 
\[\begin{tikzcd}[row sep=large]
\ntan_{k,*}\arrow[d,shift left=1ex, "\gamma"]\\
\gal_{*}\arrow[u,hook,shift left=1.4ex, "\dashv"', "(-)_{\lin}"]
\end{tikzcd}
\]
\begin{remark}
I do not have a direct categorical description of $(-)_{\lin}$ - it must behave like a closed form of the ``(Artin) motivization'' procedure used to construct motives from varieties. 
\end{remark}

\end{document}